\documentclass[bj,11pt,noinfoline]{imsart}
\usepackage{amsmath} 
\usepackage{mathrsfs}
\usepackage{amssymb} 
\usepackage{amsfonts} 
\usepackage{amsthm} 
\usepackage[numbers]{natbib}
\usepackage[colorlinks,citecolor=blue,urlcolor=blue]{hyperref}
\usepackage[margin=1in]{geometry}

\theoremstyle{plain}
\newtheorem{theorem}{Theorem}[section]

\newtheorem{lemma}[theorem]{Lemma}
\newtheorem{corollary}[theorem]{Corollary}

\theoremstyle{definition}
\newtheorem{definition}[theorem]{Definition}

\theoremstyle{remark}
\newtheorem{remark}[theorem]{Remark}
\newtheorem*{remark*}{Remark}

\numberwithin{equation}{section}

\newcommand{\bm}[1]{{\mbox{\boldmath$#1$}}}
\def\DOI#1{\href{http://dx.doi.org/#1}{DOI:#1}}

\begin{document}
\begin{frontmatter}
\title{Viscosity characterization of the explosion time distribution for diffusions}
\runtitle{Viscosity characterization of explosion time distribution}

\begin{aug}
\author{\fnms{Yinghui} \snm{Wang}\thanksref{a,e1}\ead[label=e1,mark]{yinghui@math.columbia.edu,\ yinghui@alum.mit.edu}}

\address[a]{Department of Mathematics, Columbia University, New York, NY 10027, USA.
\printead{e1}}

\runauthor{Y. Wang}

\affiliation{Columbia University}

\end{aug}

\begin{abstract}
We show that the tail distribution $U$ of the explosion time for a multidimensional diffusion (and more generally, a suitable function $\mathscr{U}$ of the Feynman-Kac type involving the explosion time) is a viscosity solution of an associated parabolic partial differential equation (PDE), provided that the dispersion and drift coefficients of the diffusion are continuous. This generalizes a result of Karatzas and Ruf (2015), who characterize $U$ as a classical solution of a Cauchy problem for the PDE in the one-dimensional case, under the stronger condition of local  H\"{o}lder continuity on the coefficients.
We also extend their result to $\mathscr{U}$ in the one-dimensional case by establishing the joint continuity of $\mathscr{U}$. Furthermore, we show that $\mathscr{U}$ is dominated by any nonnegative classical supersolution of this Cauchy problem. 
Finally, we consider another notion of weak solvability, that of the distributional (sub/super)solution, and show that $\mathscr{U}$ is no greater than any nonnegative distributional supersolution of the relevant PDE.
\end{abstract}


\end{frontmatter}

\section{Notation and previous results}

Let $\mathcal{O}$ be a fixed domain of $\mathbb{R}^n$, not necessarily bounded, and $b=(b_1,\dots,b_n): \mathcal{O} \rightarrow \mathbb{R}^n$ and $\sigma=(\sigma_{ik})_{n\times n}: \mathcal{O} \rightarrow \mathbb{M}(n)$ 
two measurable functions.
We assume throughout this paper that for every $x\in \mathcal{O}$, the stochastic differential equation
\begin{equation}\label{eq:SDE}
\text{d}X_i(t) = \sum_k \sigma_{ik}(X(t))\, \text{d}W_k(t) + b_i(X(t))\, \text{d}t\,,\ \ \ \ \ \ i = 1,\dots,n\,,\ \ X(0)=x
\end{equation}
admits a weak solution $(X (\cdot)  ,W (\cdot) )$, $(\Omega, \mathscr{F}, \mathbb{P})$, $\{\mathscr{F}(t)\}_{0\le t < \infty}$ which is {\it unique} in law and defined up until the {\it explosion time} 
\begin{equation*}
S \,:=\, \inf \big\{ t \ge 0 : X(t ) \notin \mathcal{O} \big\}\,. 
\end{equation*}
Here $\,\mathbb{M}(n)$ is the set of $n\times n$ real matrices, 
$W(\cdot)=(W_1 (\cdot),\dots,W_n (\cdot))\,$ is an $n-$dimensional Brownian motion, and the state process   $X(\cdot) =(X_1 (\cdot),\dots,X_n(\cdot))$ is strongly Markovian as a consequence of the uniqueness in law. 
In (\ref{eq:SDE}) and throughout this paper, summations extend from $1$ to $n$\,. We shall denote the distribution of the process $X(\cdot)$ starting at $x$ by $\mathbb{P}_x$\,;   the expectation  $\mathbb{E}^{\mathbb{P}_x}$ corresponding to this probability measure by $\mathbb{E}_x$\,; and by $U$ the {\it tail distribution function} of the explosion time $S$: 
\begin{eqnarray}
\label{eq:U}
[0,\infty) \times \mathcal{O} \ni (t,x) \longmapsto U(t,x) \,:=\,  \mathbb{P}_x [S > t] \in [0,1]\,.
\end{eqnarray}

In \cite{KR}, the authors studied the one-dimensional case $\, n=1\,$,  when $\mathcal{O}=(\ell,r)$ is a fixed open interval with $-\infty \le \ell < r \le \infty$ and $b: \mathcal{O} \rightarrow \mathbb{R}$ and $\sigma: \mathcal{O} \rightarrow \mathbb{R} \setminus \{0\}$ are two measurable functions satisfying
\begin{equation*}
\int_{K} \left( \frac{1}{\sigma^2(y)} + \frac{|b(y)|}{\sigma^2(y)} \right) \mathrm{d}y < \infty
\end{equation*}
for every compact set $K \subset \mathcal{O}$. 
According to the same arguments as in \cite{ES84}, \cite{ES91}
or \cite[Theorem 5.5.15]{KS}, the stochastic differential equation \eqref{eq:SDE}
admits then a weak solution defined up until the explosion time $S$, which is unique in law.  
As a result, the state process $X(\cdot)$ has the strong Markov property.  It is then shown in \cite{KR} that, in the one-dimensional case and under the above conditions,  the function $U$ of (\ref{eq:U}) is jointly continuous, and  furthermore, that if $\sigma$ and $b$ are locally H\"{o}lder continuous, then $U$ solves the Cauchy problem for the parabolic partial differential equation (PDE)
\begin{equation}\label{eq:PDE1}
u_t(t,x) - \frac{1}{2}\sigma^2(x)u_{xx}(t,x) - b(x) u_x(t,x) = 0\,,\quad (t,x) \in (0,\infty) \times \mathcal{O}
\end{equation}
with the initial condition
\begin{equation}\label{eq:initial cond}
u(0,x) = 1\,, \qquad x \in \mathcal{O},
\end{equation}
and is in fact the smallest nonnegative (super)solution of the Cauchy problem \eqref{eq:PDE1}, (\ref{eq:initial cond}). 

\smallskip
\noindent
{\it Remark:} In the one-dimensional case, the   Feller test provides necessary and sufficient conditions under which explosion occurs with positive probability, i.e., $\, \mathbb{P}_x [S<\infty] >0\,$; in this case the Cauchy problem \eqref{eq:PDE1}, (\ref{eq:initial cond}) has lots of classical solutions in addition to its trivial solution $u \equiv 1$,  of which the function $U$ in \eqref{eq:U} is the smallest. 
A host of examples is provided in \cite[Section 6]{KR}.  \qed

\smallskip
A natural question to ask then,  is whether $U$ is still a solution to  \eqref{eq:PDE1}, probably in some weak or generalized  sense, when we weaken the H\"{o}lder condition on $\sigma$ and $b$ to {\it continuity only}. The answer is affirmative. Moreover, this result can be generalized to multidimensional  settings and to functions $\mathscr{U}$ given by multidimensional {\it Feynman-Kac-type} expressions involving the explosion time, namely  
\begin{equation}\label{eq:scr{U}}
\mathscr{U}(t,x) := 
\mathbb{E}_{x} \left[\bm{1}_{\{S > t\}} f(X(t)) \exp \left(-\int_0^t h(X(s))\, \text{d}s\right)\right].
\end{equation}
Here $f$ and $h:\mathcal{O}\to \mathbb{R}$ are   measurable functions, such that   $\mathscr{U}(t,x)$ is well-defined by \eqref{eq:scr{U}} and finite for all $(t,x)\in (0,\infty)\times \mathcal{O}$\,; when  $f\equiv 1$ and $h\equiv 0$, the function $\mathscr{U}$ coincides with the tail distribution $U$ of the explosion time, as defined in \eqref{eq:U}.

\smallskip
\noindent
{\it Remark:}
A sufficient condition for the right-hand side of \eqref{eq:scr{U}} to be well-defined and finite, is that $f$ be bounded and $h$ be  bounded from below (e.g., $h\ge 0$) on $\mathcal{O}$. \qed

\smallskip
For ease of notation we introduce the continuous, adapted, strictly positive process  
\begin{equation}\label{eq:Y}
Y(t):=\exp \left(-\int_0^t h(X(s))\, \text{d}s\right),
\end{equation}
and thus $\mathscr{U}(t,x) =
\mathbb{E}_{x} \left[\bm{1}_{\{S > t\}} f(X(t)) Y(t)\right]$.
Note the dynamics 
$\,\text{d} Y(t) = -h(X(t))Y(t)\,\text{d}t$.

\subsection{Preview}

Section \ref{section:vis} introduces the definition of {\it viscosity solution} and characterizes the function $\mathscr{U}$ as a viscosity solution to the Feynman-Kac-type version of \eqref{eq:PDE1}, \eqref{eq:initial cond}, namely 
\begin{equation}\label{eq:F-K}
\left(u_t -\mathcal{L}'u\right) (t,x) = 0\,,\quad (t,x)\in (0,\infty)\times \mathcal{O} 
\end{equation}
with
\begin{equation}\label{eq:f}
u(0,x) = f(x)\,,\quad x\in \mathcal{O} \,,
\end{equation}
where we have set 
\begin{equation}\label{eq:L'}
\mathcal{L}'u(t,x)\, := \frac{1}{\,2\,}\,\sum_{i,j} a_{ij}(x)u_{x_i x_j}(t,x) + \sum_{i} b_i(x) u_{x_i}(t,x) - h(x)u(t,x)\,,\ \ (a_{ij})_{n\times n} = a:=\sigma\sigma^T.
\end{equation} 
Section \ref{section:continuity} establishes the joint continuity of $\mathscr{U}$ in its arguments $(t,x)$ in the one-dimensional case, and therefore characterizes $\mathscr{U}$ as a classical solution of  the Cauchy problem  \eqref{eq:F-K}, \eqref{eq:f}.
Section \ref{section:min} shows that $\mathscr{U}$ is dominated by {\it any}  nonnegative classical supersolution of  this Cauchy problem. 
Finally, Section  \ref{section:dist} defines another kind of weak solution, the {\it distributional (sub/super) solution},  and shows that $\mathscr{U}$ is no greater than {\it any}  nonnegative distributional supersolution of \eqref{eq:F-K}. 

\subsection{The special case $U$}

In the case of $\mathscr{U}\equiv U$, our results can be summarized as follows.

\begin{theorem}\label{thm:viscosity}
Assume that for every $x\in \mathcal{O}$, the stochastic differential equation \eqref{eq:SDE}
admits a weak solution $(X (\cdot)  ,W (\cdot) )$, $(\Omega, \mathscr{F}, \mathbb{P})$, $\{\mathscr{F}(t)\}_{0\le t < \infty}$ which is  unique in law and defined up until the explosion time $S := \inf\, \{ t \ge 0 : X(t ) \notin \mathcal{O} \}$.
Denote the distribution of the process $X(\cdot)$ starting at $x$ by $\mathbb{P}_x$.  
If the functions $a:=\sigma\sigma^T$ and $b$ are continuous, then the function
$U(t,x) :=  \mathbb{P}_x [S > t]$, $(t,x)\in [0,\infty) \times \mathcal{O} $  is a viscosity solution of the following parabolic equation:
 \begin{equation}\label{eq:PDE}
\big(u_t - \mathcal{L}u \big)(t,x) = 0,\quad (t,x) \in (0,\infty) \times \mathcal{O}
\end{equation} 
with
\begin{equation}\label{eq:L}
\mathcal{L}u(t,x)\, := \, \frac{1}{2}\,\sum_{i,j} a_{ij}(x) \, u_{x_i x_j}(t,x) + \sum_{i} b_i(x) \, u_{x_i}(t,x)\,,
\end{equation} 
and thus a viscosity solution of the Cauchy problem \eqref{eq:PDE}, \eqref{eq:initial cond}, since $U(0,\cdot)=1$. 
\end{theorem}

\begin{theorem}\label{thm:smallest supersol}
The function $\,U$ is dominated by any nonnegative, classical supersolution $\,\mathcal{U} \in C([0,\infty) \times \mathcal{O}) \cap C^{1,2}((0,\infty) \times \mathcal{O})$  of the Cauchy problem \eqref{eq:PDE}, \eqref{eq:initial cond}.
\end{theorem}

\begin{theorem}\label{thm:smallest dist supersol}
Assume that $\sigma$ and $b$ are locally bounded and $a$ is locally strictly elliptic. If $\,v\in W^{1,2}_{loc}((0,\infty) \times \mathcal{O})$ is a nonnegative distributional supersolution of \eqref{eq:PDE} with $v(0,\cdot)\geq 1$ on $\mathcal{O}$, then $v\geq U$ on $[0,\infty) \times \mathcal{O}$. 

Consequently, if the function $U$ belongs to the space $\, W^{1,2}_{loc}((0,\infty) \times \mathcal{O})\,$ and is a distributional supersolution of \eqref{eq:PDE}, then it is the smallest nonnegative distributional supersolution of \eqref{eq:PDE} that satisfies  $u(0,\cdot)\geq 1$ on $\mathcal{O}$. 
\end{theorem}

\section{Viscosity characterization of $\mathscr{U}(t,x)$}\label{section:vis}

This section develops a Feynman-Kac-type result involving the explosion time, which characterizes  the function  $(t,x) \mapsto \mathscr{U}(t,x)$ as a viscosity solution of the associated parabolic equation \eqref{eq:F-K}.

\subsection{Definition of viscosity solutions}
\label{subsection:Definition of Viscosity Solutions}

We first recall from \cite{CIL} the definition of viscosity (sub/super)solutions of a second-order parabolic PDE.  Let $\mathcal{O}$ be an open subset of $\mathbb{R}^n$ and $(t,x,y,p,q) \mapsto F(t,x,y,p,q)$ a continuous, real-valued  mapping defined on $(0,\infty)\times \mathcal{O}\times \mathbb{R}\times \mathbb{R}^n\times \mathbb{S}(n)$  and    
satisfying the ellipticity condition 
\begin{equation*}
 F(t,x,y,p,q_1) \leq F(t,x,y,p,q_2) {\mathrm{ \ \ \ whenever\ }} \ q_1 \geq q_2 
\end{equation*}
for all $(t,x,y,p) \in (0,\infty)\times \mathcal{O}\times \mathbb{R}\times \mathbb{R}^n$. Here $\mathbb{S}(n)$ is the set of $n\times n$ real symmetric matrices.

Consider the second-order parabolic PDE
\begin{equation}
\label{eq:F}
u_t(t,x) + F\left(t,x,u(t,x),Du(t,x),D^2 u(t,x)\right) = 0,\ \ \ \ (t,x)\in (0,\infty)\times \mathcal{O}
\end{equation}
with $Du=(u_{x_1}, u_{x_2}, \dots, u_{x_n})'$ and $D^2 u=(u_{x_i x_j})_{n\times n}$\,.

\begin{definition}
\label{def:viscosity sol} {\bf (i)}
We say that a function $u :(0,\infty) \times \mathcal{O}\to \mathbb{R}$ is a {\em viscosity supersolution} of the equation \eqref{eq:F}, if
\begin{equation}\label{eq:F>=0}
\varphi_t(t_0,x_0) + F\left(t_0,x_0,u_*(t_0,x_0),D\varphi(t_0,x_0),D^2 \varphi(t_0,x_0)\right)\geq 0
\end{equation}
holds for all $(t_0,x_0)\in (0,\infty) \times \mathcal{O}$ and test functions $\varphi  \in C^{1,2}((0,\infty) \times \mathcal{O})$ such that $(t_0,x_0)$ is a minimum of $\,u_*-\varphi\,$  on $(0,\infty) \times \mathcal{O}$. We have denoted by $$\, u_*(t,x) : =  \liminf_{(s,y)\rightarrow (t,x)} u(s,y),\ \ \ \ (t,x)\in (0,\infty)\times \mathcal{O}$$    the {\em lower-semicontinuous envelope} of $u$, i.e., the largest lower-semicontinuous function dominated pointwise by the function $u$.
 
\smallskip
{\bf (ii)} Similarly, a function $u(t,x):(0,\infty) \times \mathcal{O}\to \mathbb{R}$ is a {\em{viscosity subsolution}} of \eqref{eq:F}, if
\begin{equation}\label{eq:F<=0}
\varphi_t(t_0,x_0) + F\left(t_0,x_0,u^*(t_0,x_0),D\varphi(t_0,x_0),D^2 \varphi(t_0,x_0)\right)\leq 0
\end{equation}
holds for all $(t_0,x_0)\in (0,\infty) \times \mathcal{O}$ and test functions $\varphi \in C^{1,2}((0,\infty) \times \mathcal{O})$ such that $(t_0,x_0)$ is a maximum of $\,u^*-\varphi\,$  on $(0,\infty) \times \mathcal{O}$. 
We have denoted by 
\begin{equation}\label{eq:u^*}
\,u^*(t,x) := \limsup_{(s,y)\rightarrow (t,x)} u(s,y),\ \ \ \ (t,x)\in (0,\infty)\times \mathcal{O}
\end{equation}
the {\em upper-semicontinuous envelope} of $u$, i.e., the smallest upper-semicontinuous function that dominates pointwise the function $u$.
  
\smallskip

{\bf (iii)} Finally, we say that $u$ is a {\em{viscosity solution}} of \eqref{eq:F}, if it is both a viscosity supersolution and a viscosity subsolution of this equation.

\smallskip
In addition, this definition is not changed if the minimum and maximum are {\em strict and/or local}.
\end{definition} 

In our setting we have  
\begin{equation*}
F(t,x,y,p,q) \,=\, - \, \frac{1}{\,2\,}  \sum_{i,j} a_{ij}(x)\, q_{ij} - \sum_{i} b_i(x)\, p_i + h(x)\, y\,,
\end{equation*}  
and the left-hand sides of (\ref{eq:F>=0}) and (\ref{eq:F<=0}) simplify to $\big(\varphi_t-\mathcal{L}'\varphi \big)(t_0,x_0)$ with $\mathcal{L}'$ defined in \eqref{eq:L'}. Since $a=\sigma \sigma^T$ is positive-semidefinite, 
the function   $F$ satisfies the ellipticity condition. We also need $F$ to be continuous,  
which means that the functions $a_{ij}$, $b_i$ and $h$ must be continuous for all indices $i,j$.

\subsection{Viscosity characterization}

\begin{theorem}\label{thm:viscosity F-K}
Assume that the functions $a$, $b$ and $h$ are continuous. Then the function $\mathscr{U}$ of \eqref{eq:scr{U}} is a viscosity solution of the parabolic equation \eqref{eq:F-K}, 
and thus a viscosity solution of the Cauchy problem \eqref{eq:F-K}, \eqref{eq:f}, since $\mathscr{U}(0,x)=f(x)$.
\end{theorem}

We first highlight the main idea for the proof of viscosity subsolution property without many of the technicalities. Similar arguments will lead to the viscosity supersolution property.

\smallskip

We prove by contradiction, assuming the contrary of \eqref{eq:F<=0} in Definition \ref{def:viscosity sol} that there exist   $\varphi \in C^{1,2}\left( (0,\infty) \times \mathcal{O}\right)$ and  $(t_0,x_0) \in (0,\infty) \times \mathcal{O}$ such that  $(t_0,x_0)$ is a strict maximum of ${\mathscr{U}}^* - \varphi$, that the maximal value equals $0$\,, and that 
\begin{equation}\label{eq:hatG>0}
(\varphi_t-{\mathcal{L}'}\varphi)(t_0,x_0)>0\,. 
\end{equation}
It then follows from the definition \eqref{eq:u^*} of ${\mathscr{U}}^*$ that we can take a pair $(t^*,x^*)$ close to $(t_0,x_0)$ such that $(\varphi-{\mathscr{U}}) (t^*,x^*) \,(\ge 0)$ is sufficiently small, say less than a small constant $C_2$\,. 
Under $\mathbb{P}_{x^*}$, we have
\begin{equation}\label{eq:phi>E[phiY]}
\varphi(t^*,x^*) - \mathbb{E}_{x^*} \left[\varphi\big(t^*-\nu,{X}(\nu)\big)Y(\nu)\right]
=\mathbb{E}_{x^*} \left[\int_0^{\nu} \mathscr{G}\big(t^*-t,{X}(t)\big)Y(t)\, \mathrm{d}t\right] \ge 0\,,
\end{equation}
for any sufficiently small stopping time $\nu$\,, where the process $Y(\cdot)$ is defined \eqref{eq:Y}, and 
$\mathscr{G}(t,x) := (\varphi_t-\mathcal{L}'\varphi)(t,x)$, which is positive for any $(t,x)$ sufficiently close to $(t^*,x^*)$ by assumption \eqref{eq:hatG>0}.

On the other hand, in \eqref{eq:phi>E[phiY]}, we estimate $\varphi(t^*,x^*)$ from above by $C_2+\mathscr{U}(t^*,x^*)$, 
and $\varphi \big(t^*-\nu,{X}(\nu)\big)$ from below by $C_1+{\mathscr{U}}\big(t^*-\nu,{X}(\nu)\big)$ with $C_1$ a small positive constant, and deduce that
\begin{equation*}
\mathscr{U}(t^*,x^*)
>\mathbb{E}_{x^*} \left[\mathscr{U} (t^*-\nu,{X}(\nu))Y(\nu)\right].
\end{equation*}
This inequality contradicts the martingale property of the process $\mathscr{U} (t^*-\nu,{X}(\nu))Y(\nu)$,
which is a consequence of the strong Markov property of $X(\cdot)$.

\smallskip

When implementing this idea, the stopping time $\nu$ needs to be not only small, but also satisfy that on $[0,\nu]$, the process
$X(\cdot)$ is bounded and close to  $x^*$; however, the $\nu$ cannot be too small, in order to ensure that $\varphi \big(t^*-\nu,{X}(\nu)\big)\ge C_1+{\mathscr{U}}\big(t^*-\nu,{X}(\nu)\big)$ holds. These considerations inspire us to design $\nu$ as in  \eqref{eq:nu F-K} below.
\qed

\begin{proof}[Proof of Theorem \ref{thm:viscosity F-K}]
Let us show that $\mathscr{U} $ is a viscosity subsolution to \eqref{eq:F-K}. The proof for the viscosity supersolution property is similar.

According to the definition of viscosity subsolution, if suffices to verify that for any test function $\varphi \in C^{1,2}((0,\infty) \times \mathcal{O})$, and for any $(t_0,x_0) \in (0,\infty) \times \mathcal{O}$ with 
\begin{equation}
\label{eq:max F-K}
(\mathscr{U}^* - \varphi)(t_0,x_0) = 0 > (\mathscr{U}^* - \varphi)(t,x)\,, \ \ \forall\  (t,x)\in (0,\infty) \times \mathcal{O}\,,
\end{equation}
 i.e., such that $(t_0,x_0)$ is a strict maximum of $ \, \mathscr{U}^* - \varphi \,$, we have the inequality 
\begin{equation}
\label{eq:scr{G}}
\mathscr{G}(t_0,x_0) \le 0\, \qquad \text{for the function} \qquad 
\mathscr{G}(t,x) := \left(\varphi_t-\mathcal{L}'\varphi\right)(t,x)\,,
\end{equation}
with $\mathcal{L}'$ and $\mathscr{U}^*$ defined in \eqref{eq:L'} and \eqref{eq:u^*}. 
We shall argue this by contradiction, assuming that $$\mathscr{G}(t_0,x_0)>0\,.$$

Since the function $\mathscr{G} $  just introduced in  (\ref{eq:scr{G}}) is continuous, there exists a neighborhood
\begin{equation}\label{eq:N}
\mathcal{N}_{\delta} := (t_0 - \delta, t_0 + \delta) \times B_{\delta}(x_0)\subset\subset (0,\infty) \times \mathcal{O}
\end{equation} 
of $(t_0,x_0)$, on which $\mathscr{G}>0$ holds (``$\subset\subset$'' means {\it compactly contained in}). 
Recalling \eqref{eq:max F-K} and the condition that $h$ is continuous and thus locally bounded,
we then introduce the constants 
\begin{equation}
\label{eq:C_1}
C_1:= - \max_{\partial \mathcal{N}_{\delta}} \, \big(\mathscr{U}^* - \varphi \big)(t,x) > 0  \,,\ \ \ \ \ 
C_2 := C_1\, \exp\left(-2\delta \sup_{x\in B_{\delta}(x_0)}|h(x)| \right).
\end{equation}
We observe that 
$$\limsup_{(t,x)\to (t_0,x_0)}(\mathscr{U}-\varphi)(t,x) 
= (\mathscr{U}^*-\varphi)(t_0,x_0)=0 $$ 
holds by the definition \eqref{eq:u^*} of $\mathscr{U}^*$ and the continuity of $\varphi$, hence there exists $(t^*,x^*)\in \mathcal{N}_{\delta}$ such that  
\begin{equation}
\label{eq:t*,x* F-K}
 (\mathscr{U}-\varphi)(t^*,x^*)>-C_2\,.
\end{equation}

Let us   consider the stopping time
\begin{equation}
\label{eq:nu F-K}
\nu \,\big(=\nu(\omega)\big)\,:= \, \inf \big\{s\in (0, t^*] : \big(t^*-s,X(s)\big)\notin \mathcal{N}_{\delta} \big\}  \wedge t^* ,
\end{equation}
and note that  the definition \eqref{eq:N} of $\mathcal{N}_{\delta}$ implies 
\begin{equation}
\label{S2d}
\nu<S \qquad \text{and} \qquad \nu \le t^*-(t_0-\delta)=(t^*-t_0)+\delta<t^* \wedge (2\delta)\,.
\end{equation} 

Now thanks to the assumption $\varphi \in C^{1,2}((0,\infty)\times \mathcal{O})$, we can apply It\^o's change of variable rule  to $\varphi(t^*-t,X(t))Y(t)$ for $t\in [0, \nu]\subset[0,t^*)$ and plug in \eqref{eq:SDE} to derive the semimartingale decomposition 
\begin{equation}\label{eq:dIto F-K}
\text{d}\big[ \varphi(t^*-t,X(t))Y(t)\big]
= - \mathscr{G}(t^*-t,X(t))Y(t)\, \text{d}t + \sum_{i,k} \varphi_{x_i}(t^*-t,X(t))\,\sigma_{ik}(X(t))Y(t)\, \text{d}W_k(t)\,, \ \ 
\end{equation} 
with the process $Y(\cdot)$ defined in \eqref{eq:Y} and the function $\mathscr{G}$ in  \eqref{eq:scr{G}} (see Appendix \ref{app:phiY} for a proof).

Integrating (\ref{eq:dIto F-K}) with respect to $t$ over $[0,\nu]$ and taking the expectation under $\mathbb{P}_{x^*}$ yields
\begin{equation}\label{eq:contradiction F-K}
\varphi(t^*,x^*) - \mathbb{E}_{x^*}\left[\varphi\big(t^*-\nu,X({\nu})\big)Y(\nu)\right]= \mathbb{E}_{x^*}\left[\int_{0}^{\nu} \mathscr{G}(t^*-t,X(t))\,Y(t)\, \text{d}t\right] \ge0 \,.
\end{equation} 
Here the last inequality comes from the assumption $\mathscr{G} >0$ on $\mathcal{N}_{\delta}$\,,
whereas the equality holds because the expectations of the integrals with respect to $\mathrm{d}W_k(t)$ have all vanished, due to: \\
{\bf (1)} the
uniform boundedness of $Y(t)$ on $[0,\nu]$;
in fact, we have $Y(t)\le \exp\left(t \sup_{x\in B_{\delta}(x_0)}|h(x)| \right) \le\exp\left(2\delta \sup_{x\in B_{\delta}(x_0)}|h(x)| \right)$ for all $0\le t\le \nu< 2\delta$\,; and \\
{\bf (2)} the boundedness of  $\varphi_{x_i}$ on $\mathcal{N}_{\delta}$, and $a_{ij}$ and thus $\sigma_{ik}$ on $B_{\delta}(x_0)$ (recalling that $a=\sigma\sigma^T$).

\smallskip
Combining \eqref{eq:contradiction F-K} with \eqref{eq:C_1} and \eqref{eq:t*,x* F-K} leads to 
\begin{align} \label{eq:<0 F-K}
\nonumber 0 \,\le&\ \, \varphi(t^*,x^*) - \mathbb{E}_{x^*}\left[\varphi\big(t^*-\nu,X({\nu})\big)Y(\nu)\right]
< C_2 + \mathscr{U}(t^*,x^*) - \mathbb{E}_{x^*}\left[\big(C_1+\mathscr{U}\big(t^*-\nu,X({\nu})\big)\big)Y(\nu)\right]
\\
=&\ \left\{ C_2- C_1\,\mathbb{E}_{x^*}\left[Y(\nu)\right]\right\} +\left\{ \mathscr{U}(t^*,x^*) - \mathbb{E}_{x^*}\left[\mathscr{U}\big(t^*-\nu,X({\nu})\big)Y(\nu)\right]\right\}.
\end{align}
We see that the first term on the right-hand side is at most zero by the definition \eqref{eq:C_1} of $C_2$ and the definition \eqref{eq:Y} of $Y(\cdot)$, along with  the fact  $\nu<2\delta$ from \eqref{S2d}.
Thus we will arrive at a contradiction as soon as we have shown that the second term on the right-hand side of \eqref{eq:<0 F-K} equals zero, namely
\begin{equation}\label{eq:mart}
\mathbb{E}_{x^*} \left[\mathscr{U}(t^*-\nu,X(\nu))Y(\nu)\right]=\mathscr{U}(t^*,x^*) \,.
\end{equation}

In fact, it follows from the strong Markov property of $X(\cdot)$ that
\begin{align*}
\mathscr{U}(t^*-\nu,X(\nu))=
& \left.\mathbb{E}_{\xi} \left[\bm{1}_{\{S > t^*-\nu\}} f(X(t^*-\nu)) \exp \left(-\int_0^{t^*-\nu} h(X(s))\, \text{d}s\right)\right]\right|_{\xi=X(\nu)}\\
=&\ \mathbb{E}_{x^*} \left[\bm{1}_{\{S > t^*\}} f(X(t^*)) \exp \left(-\int_{\nu}^{t^*} h(X(s))\, \text{d}s\right)\left.\vphantom{\int_0^{t^*-\nu}}\right|\mathscr{F}(\nu)\right],\quad  \mathbb{P}_{x^*}\text{-a.s.},
\end{align*}
thus plugging into the left-hand side of \eqref{eq:mart} along with the  definition \eqref{eq:Y} of $Y(\cdot)$ yields
\begin{align*}
&\ \mathbb{E}_{x^*} \left[\mathbb{E}_{x^*} \left[\bm{1}_{\{S > t^*\}} f(X(t^*)) \exp \left(-\int_{\nu}^{t^*} h(X(s))\, \text{d}s\right)\left.\vphantom{\int_0^{t^*-\nu}}\exp \left(-\int_0^{\nu}h(X(s))\, \text{d}s\right)\right|\mathscr{F}(\nu)\right] \right]\\
=&\ \mathbb{E}_{x^*} \left[\bm{1}_{\{S > t^*\}} f(X(t^*)) \exp \left(-\int_{0}^{t^*} h(X(s))\, \text{d}s\right)\right] 
= \mathscr{U}(t^*,x^*)\,. \qedhere
\end{align*}
\end{proof}

\section{Joint continuity of $\,\mathscr{U}$ in the one-dimensional case}
\label{section:continuity}

In this section, we fix $n=1$. The function $U(t,x)$ is shown to be jointly continuous in $(t,x)$ on $[0,\infty)\times \mathcal{O}$  \cite[Proposition 4.3]{KR}. This section generalizes this result for $\,\mathscr{U}(t,x)$ by starting with the continuity in $t$ and then using a ``coupling" argument in conjunction with \cite[(4.4)]{KR}.

Thanks to the joint continuity of $\,\mathscr{U}(t,x)$, arguments similar to those in \cite[Lemma 5.1 and  Proposition 5.2]{KR} show that $\mathscr{U}$ is actually a classical solution of the Cauchy problem \eqref{eq:F-K}, \eqref{eq:f}, whenever the function $f$ is continuous and functions $a>0$, $b$ and $h\ge 0$ are locally H\"older continuous. This result  extends \cite[Proposition 5.2]{KR} to the more general Feynman-Kac context (see Corollary \ref{coro:continuity}).

\begin{theorem}\label{thm:continuity}
When $\,n=1$, the function $\,\mathscr{U}(t,x)$ is jointly continuous in $(t,x)$ on $[0,\infty)\times \mathcal{O}$ if $\,f$ and $\,h$ are bounded, $a>0$ and either of the following conditions holds:

{\bf (i)} $f$ is H\"{o}lder continuous and $a$ and $b$ are bounded;

{\bf (ii)} $f\in C^2(\mathcal{O})$ and the functions $f'\sigma$ and $f'b+ \frac{1}{\,2\,}f''a$ are bounded on $\mathcal{O}$.
\end{theorem}

\begin{remark}
The tail distribution function $U(t,x)$ with $f\equiv 1$ and $h\equiv 0$ is a special case that satisfies Condition (ii).
\end{remark}

\begin{corollary}\label{coro:continuity}
Under the assumptions of Theorem \ref{thm:continuity}, if the functions $a$, $b$ and $h\ge 0$ are locally H\"older continuous, then $\,\mathscr{U}(t,x)$ is a classical solution of the Cauchy  problem \eqref{eq:F-K}, \eqref{eq:f}.
\end{corollary}

\begin{proof}[Proof of Theorem \ref{thm:continuity}]
Let $C_0>0$ be an upper bound of $|f|$ and $|h|$ on $\mathcal{O}$, then 
\begin{equation}\label{eq:Y<=}
0<Y(t)\le e^{t\,C_0},\quad \forall\ t\in [0,\infty)\,.
\end{equation} 

We first prove the continuity of $\,\mathscr{U}(t,x)$ in $t$.

\begin{lemma}\label{lemma:continuous in t}
The function $\,t \mapsto \mathscr{U}(t,x)\,$  is continuous for any given $x\in \mathcal{O}$.
\end{lemma}
\begin{proof}
Fix $(t,x)\in [0,\infty)\times \mathcal{O}$. 
Let us show the left-continuity. 
Switching $t$ and $t'$ below will give a proof for the right-continuity. 
Assume that $t\in(0,\infty)$. 
For any $t'\in [0,t)$, by definition we have
\begin{align*}\nonumber
\mathscr{U}(t',x) - \mathscr{U}(t,x)
=&\ \mathbb{E}_{x} \left[\bm{1}_{\{S > t'\}} f(X(t')) Y(t')\right] - \mathbb{E}_{x} \left[\bm{1}_{\{S > t\}} f(X(t)) Y(t)\right]\\ \nonumber
=&\  \mathbb{E}_{x} \left[\left(\bm{1}_{\{S > t'\}} - \bm{1}_{\{S > t\}}\right) f(X(t')) Y(t')\right]
+ \mathbb{E}_{x} \left[\bm{1}_{\{S > t\}} f(X(t')) (Y(t') - Y(t))\right]\\
&\ +\, \mathbb{E}_{x} \left[\bm{1}_{\{S > t\}}  (f(X(t')) - f(X(t))) Y(t)\right]\\ \nonumber
=:&\ \Delta_1 +  \Delta_2 + \Delta_3 \,.
\end{align*}
It then suffices to show that each $\Delta_j \rightarrow 0$ as $t'\uparrow t$\,.

\smallskip

{\bf Case 1.} $j=1$. 
Since $\bm{1}_{\{S > t\}}$ is decreasing in $t$ and $U$ is continuous, we obtain
\begin{equation*}
|\Delta_1| \le   C_0\, e^{t\,C_0}\,\mathbb{E}_{x} \left[\bm{1}_{\{S > t'\}} - \bm{1}_{\{S > t\}}\right]
= C_0\, e^{t\,C_0} [(U(t',x)-U(t,x)] \rightarrow 0\,,\ \  \text{as}\ \  t'\uparrow t\,.
\end{equation*}  

{\bf Case 2.} $j=2$. 
For any $t<S$, we have
$$|Y(t') - Y(t)|=\left|Y(t)\left[\exp\left(\int_{t'}^t h(X(s))\, \text{d}s\right)-1\right]\right|
\le e^{t\,C_0}\left( e^{(t-t')\,C_0}-1\right),$$
since
$$ e^{(t-t')\,C_0}-1\ge \exp\left(\int_{t'}^t h(X(s))\,
 \text{d}s\right)-1
\ge e^{-(t-t')\,C_0}-1 \ge -\left( e^{(t-t')\,C_0}-1\right)
$$
(as $z + 1/z \ge 2$ for all $z>0$).
Hence
\begin{equation*}
|\Delta_2|\le  C_0 \,\mathbb{E}_{x}\big[\bm{1}_{\{S > t\}}|Y(t') - Y(t)|\big]\le C_0\, e^{t\,C_0}\left( e^{(t-t')\,C_0}-1\right)  \rightarrow 0\,,\ \  \text{as}\ \  t'\uparrow t\,.
\end{equation*}

\smallskip
{\bf Case 3.} $j=3$. {\bf (i)}
If Condition (i) holds, there exist constants $\alpha\in  (0,1]$ and $C_{\alpha},C>0$ such that $|f(x)-f(y)|\le C_{\alpha}|x-y|^{\alpha}$ and $a(x), |b(x)|\le C$ for all $x,y\in \mathcal{O}$.  
Thus H\"{o}lder's inequality gives
\begin{equation}\label{eq:Delta_3(i)}
|\Delta_3|\le e^{t\,C_0}\,\mathbb{E}_{x} \big[\bm{1}_{\{S > t\}}|f(X(t))- f(X(t'))|\big]
\le e^{t\,C_0}C_{\alpha}\,\mathbb{E}_{x} \big[|\Xi(t')|^{\alpha}\big]
\le e^{t\,C_0}C_{\alpha}\left(\mathbb{E}_{x} \left[|\Xi(t')|^2\right]\right)^{\alpha/2},
\end{equation}
where $\Xi(s):=\bm{1}_{\{S > t\}}(X(t) - X(s\wedge t))$, $s\in[0,\infty)$.
Let us recall from \eqref{eq:SDE} that
\begin{equation*}
\Xi(s) =\bm{1}_{\{S > t\}} \int_s^t \sigma (X(\varsigma))\, \text{d}W(\varsigma) + \bm{1}_{\{S > t\}}  \int_s^t b(X(\varsigma))\, \text{d}\varsigma
=: M(s) + A(s)\,,\quad \forall\ s\le t\,,
\end{equation*}
therefore Cauchy-Schwarz inequality and   Burkholder-Davis-Gundy inequalities give
\begin{equation} \label{eq:BDG}
\mathbb{E}_{x} \left[|\Xi(t')|^2\right]
\le 2\,\mathbb{E}_{x} \big[|M(t')|^2\big]
+2\,\mathbb{E}_{x}\big[|A(t')|^2\big]
\le 8\,\mathbb{E}_{x} \big[\langle M\rangle(t') \big]+2\,\mathbb{E}_{x}\big[|A(t')|^2\big].
\end{equation}

Since $\sigma^2=a\le C$ and $|b|\le C$ on $\mathcal{O}$, we obtain
\begin{equation}\label{eq:<M>}
\langle M\rangle(t') =\bm{1}_{\{S > t\}} \int_{t'}^t \sigma^2 (X(\varsigma))\, \text{d}\varsigma
\le (t-t')\, C\quad \text{and}\quad |A(t')| \le (t-t')\, C\,.
\end{equation}
Substituting into \eqref{eq:BDG} and then into \eqref{eq:Delta_3(i)}  yields 
\begin{equation}\label{eq:Delta_3}
|\Delta_3|\le e^{t\,C_0}\,C_{\alpha}\left[ 8(t-t')\, C+2(t-t')^2\, C^2\right] ^{\alpha/2}\rightarrow 0\,,\quad \text{as}\ \  t'\uparrow t\,.
\end{equation}

{\bf (ii)}
If Condition (ii) holds, let $C'>0$ be an upper bound of $(f'\sigma)^2$ and $\left|f'b+ \frac{1}{\,2\,}f''a\right|$ on $\mathcal{O}$. 
Our idea is similar to that of (i) with the process $\Xi(s)$ replaced by
$\Xi'(s):=\bm{1}_{\{S > t\}}(f(X(t)) - f(X(s\wedge t)))$, $s\in[0,\infty)$, the exponent $\alpha$ set to $1$ and the constant $C_{\alpha}$ removed. 
More precisely, we have
\begin{equation*}
|\Delta_3|\le e^{t\,C_0}\,\mathbb{E}_{x} \big[\bm{1}_{\{S > t\}}|f(X(t))- f(X(t'))|\big]
= e^{t\,C_0}\,\mathbb{E}_{x} \big[|\Xi'(t')|\big]
\le e^{t\,C_0}\left(\mathbb{E}_{x} \left[|\Xi'(t')|^2\right]\right)^{1/2}.
\end{equation*}
For any $0\le s\le t<S$, we apply It\^o's change of variable rule to $f(X(\varsigma))$, $\varsigma\in[s,t]$ and plug in \eqref{eq:SDE} to get
\begin{equation*}
\Xi'(s)=\bm{1}_{\{S > t\}} \int_{s}^t  (f'\sigma)(X(\varsigma))\, \text{d}W(\varsigma) + \bm{1}_{\{S > t\}} \int_{s}^t\!  \left( f'b+ \frac{1}{\,2\,}f''a\right)\! (X(\varsigma))\, \text{d}\varsigma
=: M'(s) + A'(s)\,.
\end{equation*}
Thus \eqref{eq:BDG}--\eqref{eq:Delta_3} hold for $\Xi', M', A'$ and $C'$, with $\sigma$  replaced by $f'\sigma$ in \eqref{eq:<M>}, $C_{\alpha}$ removed and $\alpha=1$.

\medskip
Switching $t$ and $t'$ in the above argument proves the right-continuity. More precisely, we summarize the above estimates as the following inequality that holds for all $t,t'\in[0,\infty)$  to complete the proof:
\begin{align}\label{eq:cont in t}
|\mathscr{U}(t',x) - \mathscr{U}(t,x)|
\le&  \ e^{(t\vee t')\,C_0} \left\{C_0\, |U(t',x)-U(t,x)|
+C_0 \big(e^{|t-t'|\,C_0}-1\big)\right.\\ \nonumber
&  +\! \left.\left[C_{\alpha}\!\left( 8|t-t'|\, C+2(t-t')^2 C^2\right)^{\alpha/2}\right]\! \vee\! \left[ \left(8|t-t'|\, C'\!+2(t-t')^2 C'^2\right)  ^{1/2}\right]\right\}\\  \nonumber\ 
=:&\ e^{(t\vee t')\,C_0}\, C_0\, |U(t',x')-U(0,x')| +\gamma(t,t') 
\rightarrow 0\,,\quad \text{as}\ \  t'\rightarrow t\,. \qedhere
\end{align}
\end{proof}

Now let us go back to the joint continuity.
Define the stopping time 
$$H_x:=\inf\, \{t\ge 0: X(t)=x\}\,,\quad x\in \mathcal{O}.$$

Fix $(t,x)\in [0,\infty)\times \mathcal{O}$. If $t=0$, then for any $(t',x')\in [0,\infty)\times \mathcal{O}$, it follows from \eqref{eq:cont in t} that
\begin{align*}
|\mathscr{U}(t',x') - \mathscr{U}(0,x)|
\le&\  |\mathscr{U}(t',x') - \mathscr{U}(0,x')| 
+ |\mathscr{U}(0,x') - \mathscr{U}(0,x)|\\
\le&\ e^{t'C_0}\, C_0\, |U(t',x')-U(0,x')|+\gamma(0,t')+|f(x')-f(x)|
\end{align*}
Since $f$ is continuous under either Condition (i) or (ii),  
$\gamma(0,t')\to 0$ as $t'\to 0$, and
$$ |U(t',x')-U(0,x')|
\le  |U(t',x')-U(0,x)|
+  |U(0,x)-U(0,x')|\rightarrow 0\,, \ \ \text{as}\ \ (t',x')\rightarrow (0,x)$$
by the joint continuity of $U$, we conclude that 
$ |\mathscr{U}(t',x') -\mathscr{U}(0,x)|\rightarrow 0\,,$ as $(t',x')\rightarrow (0,x)$.

\smallskip
Now assume $t>0$. For any $\varepsilon>0$, according to Lemma \ref{lemma:continuous in t}, there exists $\delta\in(0,t/2)$ such that 
\begin{equation}\label{eq:comp}
 |\mathscr{U}(\theta,x)-\mathscr{U}(t,x)|<\varepsilon\,, \quad \forall ~~\theta \in (t-2\delta, t+2\delta)\,. 
\end{equation}
For any $t'\in (t-\delta, t+\delta)$ (thus $t'>\delta$), we decompose $\mathscr{U}(t,x)$ and  $\mathscr{U}(t',x')$  into  two terms, respectively:
$$\mathscr{U}(t,x)=\mathbb{E}_{x'} \left[\bm{1}_{\{H_x<\delta\}}\mathscr{U}(t,x)\right] 
+ \mathbb{E}_{x'} \left[\bm{1}_{\{H_x\ge \delta\}}\mathscr{U}(t,x) \right]=:E_1 + E_2\,,$$
$$\mathscr{U}(t',x')
= \mathbb{E}_{x'} \left[\bm{1}_{\{S > t'\}} \bm{1}_{\{H_x<\delta\}}f(X(t')) Y(t')\right] 
+ \mathbb{E}_{x'} \left[\bm{1}_{\{S > t'\}}\bm{1}_{\{H_x\ge \delta\}} f(X(t')) Y(t')\right]=:E'_1 + E'_2\,.$$
Then using the well-known property that $\lim_{x'\rightarrow x}\mathbb{P}_{x'}[H_x\ge \delta]=0$ (a simple proof of this property can be found in \cite[(4.4)]{KR}; see also \cite[Section 3.3]{IM}), we obtain
$$|E_2| =|\mathscr{U}(t,x)|\, \mathbb{P}_{x'}[H_x\ge \delta]\,\rightarrow 0\,,\quad |E'_2| \le C_0\, e^{(t+\delta)\,C_0}\,\mathbb{P}_{x'}[H_x\ge \delta]\rightarrow 0\,,\quad \text{as}\ \ x'\rightarrow x\,.$$

Finally, the strong Markov property of $X(\cdot)$ implies that 
\begin{align*}
E'_1 =&\ \mathbb{E}_{x'} \left[\bm{1}_{\{H_x<\delta\}}\,\mathbb{E}_{x'}\left[\left.\bm{1}_{\{S > t'\}}f(X(t')) Y(t')\,\right|\mathscr{F}(H_x)\right] \right]\\
=&\ \mathbb{E}_{x'}\left[\bm{1}_{\{H_x<\delta\}}\,\mathbb{E}_{x}\left[\bm{1}_{\{S > t'-H_x\}}f(X(t'-H_x)) Y(t'-H_x)\right]\right]  \\
=&\ \mathbb{E}_{x'}\left[\bm{1}_{\{H_x<\delta\}}\,\mathscr{U}(t'-H_x,x)\right].
\end{align*}
Since $t'-H_x\in (t-2\delta,t+2\delta)$ holds on  $\{H_x<\delta\}$, on the strength of \eqref{eq:comp} we arrive at
\begin{equation*}
|E'_1-E_1|
=\mathbb{E}_{x'}\left[\bm{1}_{\{H_x<\delta\}}\left( \mathscr{U}(t'-H_x,x)-\mathscr{U}(t,x)\right) \right]
\le \mathbb{E}_{x'}\left[\bm{1}_{\{H_x<\delta\}}|\mathscr{U}(t'-H_x,x)-\mathscr{U}(t,x)|\right] < \varepsilon\,.
\end{equation*}
Hence $\mathscr{U}(\cdot\,,\cdot)$ is jointly continuous at $(t,x)$.
\end{proof}

\begin{remark}
We have a simpler proof for the right-continuity of $\,\mathscr{U}(t,x)$ in $t$, by taking advantage of the right-continuity of $\bm{1}_{\{S >\, \cdot\,\}}$ in conjunction with the continuity and boundedness of the paths of  $f(X(\cdot))$ and $Y(\cdot)$ (see \eqref{eq:Y<=}). It follows from the dominated convergence theorem that
\begin{equation*}
\qquad \lim_{t'\downarrow\, t}\mathscr{U}(t',x) =
\lim_{t'\downarrow\, t} \mathbb{E}_{x} \left[\bm{1}_{\{S > t'\}} f(X(t')) Y(t')\right]
= \mathbb{E}_{x} \left[\lim_{t'\downarrow\, t} \bm{1}_{\{S > t'\}} f(X(t')) Y(t')\right]
= \mathscr{U}(t,x)\,. \quad \qed
\end{equation*} 
\end{remark}

\section{Minimality}
\label{section:min}

For the one-dimensional case, the function $U $ of \eqref{eq:U} is dominated by every nonnegative classical supersolution of the Cauchy problem \eqref{eq:PDE}, \eqref{eq:initial cond} \cite[Proposition 5.3]{KR}. Therefore, whenever the function $U $  is known in advance to be a classical solution of this problem (for instance, when the functions $a$ and $b$ are H\"older continuous), it is also the smallest nonnegative classical supersolution of this problem. 
This section proves these results in the more general Feynman-Kac case. Our minimality results (Theorem \ref{thm:smallest supersol F-K} and Corollary \ref{coro:classical} below) have similarities -- at least in spirit --  to those in Proposition 5.3 of \cite{KR}, Problem 3.5.1 of \cite{McK}, in Exercise 4.4.7 of \cite{KS}, as well as to \cite{FK10}, \cite{FK11} and \cite{R}. 

\begin{theorem}
\label{thm:smallest supersol F-K}
Assume that the function $f$ is nonnegative on $\,\mathcal{O}$.
Then the function $\mathscr{U} $ defined in \eqref{eq:scr{U}} is dominated by every nonnegative, classical supersolution $\,\mathcal{U}(t,x)\in C([0,\infty) \times \mathcal{O}) \cap C^{1,2}((0,\infty) \times \mathcal{O})$  of the Cauchy problem \eqref{eq:F-K}, \eqref{eq:f}.
\end{theorem}

Combining Theorem \ref{thm:smallest supersol F-K} with Corollary \ref{coro:continuity} leads directly to the following corollary.

\begin{corollary}\label{coro:classical}
If the assumptions of Theorem \ref{thm:smallest supersol F-K} and Corollary \ref{coro:continuity} hold, then $\mathscr{U}$ is the smallest nonnegative classical (super)solution of the Cauchy problem  \eqref{eq:F-K}, \eqref{eq:f}. 
\end{corollary}

\begin{proof}[Proof of Theorem \ref{thm:smallest supersol F-K}]
It follows from definition \eqref{eq:scr{U}} and the supersolution property that $ \mathscr{U}(0,x)$ $ = f(x)\le \mathcal{U}(0,x)$ holds on $\mathcal{O}$. Let $\mathcal{O}_m\subset\subset \mathcal{O}$ (i.e., $\overline{\mathcal{O}_m}\subset \mathcal{O}$) be a sequence of bounded domains in $\mathbb{R}^n$ such that $\mathcal{O}_m \uparrow \mathcal{O}$ and define the stopping times \begin{equation}\label{eq:S_m F-K}
S_m:=\inf\, \{ t\ge 0: X(t)\notin \mathcal{O}_m\},\quad m \in \mathbb{N}\,.
\end{equation}
Then $S_m \uparrow S$.
 
Let us fix a positive integer $m$. For any given $(T,x)\in(0,\infty)\times \mathcal{O}_m$\,,  thanks to the assumption \, $\mathcal{U} \in C^{1,2}((0,\infty) \times \mathcal{O})$, we also have (\ref{eq:dIto F-K}) with $\varphi$ replaced by\, $\mathcal{U}$, $t^*$ by $T$, and $t$ by $t\wedge S_m\,$, $0\le t \le T$. This way we deduce from the supersolution property that\, $\mathcal{U}_t-\mathcal{L}'\,\mathcal{U}\ge 0$ holds on $(0,\infty) \times \mathcal{O}$, 
therefore the process  $\,\mathcal{U}(T-t\wedge S_m\,,\, X(t\wedge S_m)) \, Y(t\wedge S_m),$ $ 0\le t\le T\,$ is a local $\mathbb{P}_x$-supermartingale. 

This local supermartingale is nonnegative and hence a true supermartingale. Therefore we obtain
\begin{eqnarray*}
\mathcal{U}(T,x) 
&\ge& \mathbb{E}_x \big[\,\mathcal{U}(T-T\wedge S_m,X(T\wedge S_m))Y(T\wedge S_m)\big] 
\ge \mathbb{E}_x \big[ \bm{1}_{\{S_m> T\}}\,\mathcal{U}(0,X(T)) Y(T)\big]\\
&\ge&  \mathbb{E}_x \big[ \bm{1}_{\{S_m> T\}}f(X(T)) Y(T)\big]
\end{eqnarray*}
by optional sampling. 
The last expression converges to 
$\mathbb{E}_x \left[\bm{1}_{\{S> T\}}f(X(T)) Y(T)\right] = \mathscr{U}(T,x)$
as $m\rightarrow \infty$ by the monotone convergence theorem, so we conclude $\, \mathcal{U}(T,x) \ge \mathscr{U}(T,x)\,$.
\end{proof}

\section{The distributional solution}
\label{section:dist}

In this section we define and study yet another kind of weak solution, the {\it distributional solution.} We shall adopt the following definition from \cite{K}.
Set   
$$ ||f||_{C(\overline{Q})} := \sup_{y \in \overline{Q}} |f(y)|\,,\ ||f||_{W^{1,2}(\overline{Q})} := ||f_t||_{L^{n+1}(\overline{Q})} + \sum_i  ||f_{x_i}||_{L^{n+1}(\overline{Q})} + \sum_{i,j} ||f_{x_i x_j}||_{L^{n+1}(\overline{Q})} + ||f||_{C(\overline{Q})}\,.$$ 
\begin{definition}\label{def:W^{1,2}}{\em{(\cite[Definition 2.1.1]{K})}}
Let $Q$ be a bounded domain of $\mathbb{R}^{n+1}$. We denote by $\,W^{1,2}(Q)\,$ the space of functions $\, u: \overline{Q} \rightarrow \mathbb{R}\,$,  for each of which  there exists a sequence of functions $u^{(m)}\in C^{1,2}(\overline{Q})$ such that $||u - u^{(m)}||_{C(\overline{Q})} \rightarrow 0$ and $||u^{(m)} - u^{(m')}||_{W^{1,2}(\overline{Q})} \rightarrow 0$ as $m,m' \rightarrow \infty\,$. 
\end{definition} 

The continuity of the functions $u^{(m)}$  implies that $u \in C(\overline{Q})$. In addition, every function $u \in W^{1,2}(Q)$ possesses generalized (a.k.a. ``weak") derivatives $u_t$, $u_{x_i}$ and $u_{x_i x_j}$ on $Q$, which are unique almost everywhere (\cite[Definition 2.1.2]{K}).

We generalize Definition \ref{def:W^{1,2}} for unbounded domains as follows.
\begin{definition}\label{def:W^{1,2}_{loc}}
Let $Q'$ be a domain of $\mathbb{R}^{n+1}$, not necessarily bounded. Define $$W^{1,2}_{loc}(Q') \,:=\, \big\{ u: Q' \rightarrow \mathbb{R}\ |\ u \in W^{1,2}(Q) \text{ for\ all\ bounded\ domain}\ Q\subset Q' \big\}\,.$$
\end{definition} 

For $u \in W^{1,2}_{loc}(Q')$, we have $u \in C(\overline{Q})$ for all bounded domain $Q \subset Q'$,  hence $u \in C(\overline{Q'})$ as well. Further, the generalized derivatives $u_t$, $u_{x_i}$ and $u_{x_i x_j}$ can be shown to be well-defined and unique a.e. To see this, take a sequence of bounded domains $Q_m\subset Q'$ such that $Q_m \uparrow Q'$ and use the fact that the generalized derivatives are uniquely defined  a.e. on each $Q_m\,$.

\medskip
Now we are able to define the distributional (sub/super)solution. Let $Q'$ be a domain of $\mathbb{R}^{n+1}$ and $F(z,y,p,q)$ a continuous map from $Q'\times \mathbb{R}\times \mathbb{R}^n \times \mathbb{S}(n)$ to $\mathbb{R}$ satisfying the ellipticity condition:
\begin{equation*}
F(z,y,p,q_1) \leq F(z,y,p,q_2) ~~{\mathrm{ \ whenever\ }} ~q_1 \geq q_2\,,
\end{equation*}
for all $(z,y,p) \in Q'\times \mathbb{R}\times \mathbb{R}^n$ (by analogy with the $F$ in Section \ref{subsection:Definition of Viscosity Solutions}). 

\smallskip
Consider a second-order parabolic PDE
\begin{equation}\label{eq:F on Q'}
u_t(t,x) + F\left((t,x),u(t,x),Du(t,x),D^2 u(t,x)\right) = 0\,,\ (t,x)\in Q'.
\end{equation}

\begin{definition}\label{dist sol}
We say that $u \in W^{1,2}_{loc}(Q')$ is a {\em{distributional sub(super)-solution}} of \eqref{eq:F on Q'} if the generalized derivatives $u_t$, $Du$ and $D^2u$ can be chosen such that 
\begin{equation}\label{eq:F<=0 on Gamma}
u_t (t,x) + F\left((t,x),u(t,x),Du(t,x),D^2 u(t,x)\right) \leq 0\ \ (\geq 0)\,,\qquad \forall\ (t,x)\in \Gamma
\end{equation}
holds for some $\Gamma \subset Q'$ with meas\,$(Q'\setminus \Gamma)=0$\,.

A {\em{distributional solution}} is a function that is both a distributional subsolution and a distributional supersolution.
\end{definition} 

In our setting, we have $Q'=(0,\infty) \times \mathcal{O}$ and $$F((t,x),y,p,q) = -\frac{1}{\,2\,}\sum_{i,j}a_{ij}(x)\,q_{ij} - \sum_i b_i(x)\, p_i + h(x)\, y\,,$$ so the left-hand side of (\ref{eq:F<=0 on Gamma})   simplifies to $(u_t-\mathcal{L}'u)(t,x)$ (recall the operator $\mathcal{L}'$ from \eqref{eq:L'}). We see that $F$ satisfies the ellipticity condition since $a=\sigma\sigma^T$ is positive-semidefinite.

\subsection{Domination by nonnegative distributional supersolutions}

We have the following minimality result for $\mathscr{U}(t,x)$.

\begin{theorem}\label{thm:smallest dist supersol F-K}
Assume that $\sigma$, $b$ and $h$ are locally bounded and $a$ is locally strictly elliptic on $\mathcal{O}$. If $\,v\in W^{1,2}_{loc}((0,\infty) \times \mathcal{O})$ is a nonnegative distributional supersolution of \eqref{eq:F-K} with 
\begin{equation}\label{eq:>=f}
v(0,x)\geq f(x)\,, \quad x\in \mathcal{O}\,,
\end{equation}
then $v\geq \mathscr{U}$ on $[0,\infty) \times \mathcal{O}$. 

Thus, if the function $\mathscr{U}$ belongs to the space $\, W^{1,2}_{loc}((0,\infty) \times \mathcal{O})\,$ and is a distributional supersolution of \eqref{eq:F-K}, then it is the smallest nonnegative distributional supersolution of \eqref{eq:F-K} that satisfies \eqref{eq:>=f}. 
\end{theorem}

\begin{proof}
Set the generalized derivatives $v_t$, $v_{x_i}$ and $v_{x_i x_j}$ to be identically equal to zero  on $((0,\infty) \times \mathcal{O}) \setminus \Gamma$ (with $\Gamma$ as defined in Definition \ref{dist sol}). Fix an arbitrary point  $x_0\in \mathcal{O}$ and real number $T>0$\,.
Denote $\,Q_m=(0,T)\times \mathcal{O}_m\,$, with $\mathcal{O}_m$ as in the proof of Theorem \ref{thm:smallest supersol F-K}.

\smallskip
Let $u(t,x)=v(T-t,x)$, $(t,x) \in [0,T] \times \mathcal{O}$, then $u\in W^{1,2}(Q_m)$ is a distributional subsolution of $(u_t+\mathcal{L}'u)(t,x)=0$ on $Q_m$. The following result is well known. 

\begin{lemma}
\label{lemma:generalized Ito F-K}{\em{(\cite[Theorem 2.10.1]{K})}}
Let $Q$ be a bounded domain of $\,\mathbb{R}^{n+1}$ and $u \in W^{1,2}(Q)$. For any $\, (s_0, x_0) \in Q\,$ and any stopping time $\tau \leq \tau(Q,s_0) := \inf\, \{t>0: (s_0+t,X(t)) \notin Q\}$, we have
\begin{equation*}
u(s_0,x_0) = \mathbb{E}_{x_0}\left[Y(\tau) u(s_0+\tau,X(\tau)) - \int_{0}^{\tau} Y(t) (u_t+\mathcal{L}'u)(s_0+t,X(t))\,\mathrm{d}t\right],
\end{equation*}
where $X(\cdot)$ is the diffusion of \eqref{eq:SDE} starting at $\, X(0) = x_0\,$. 
\end{lemma}
Applying the above lemma with $Q=Q_m\,$, $s_0=1/i$ and $\tau=\tau_{i,m}:=\big(T-( 1 / i ) \big) \wedge S_m= \tau(Q,s_0)$ 
with $S_m$ defined in \eqref{eq:S_m F-K}, and combining with the distributional subsolution property of $u$ leads to
\begin{equation}\label{eq:u(1/i,x_0)}
u\left( \frac{1}{i} \,,x_0\right)
\geq \mathbb{E}_{x_0}\left[Y(\tau_{i,m}) \,u\left(\frac{1}{i}+\tau_{i,m},X(\tau_{i,m})\right)\right].
\end{equation}
Recall that $v$ is continuous (see the explanation right below Definition \ref{def:W^{1,2}_{loc}}), and therefore so is $u$. Thus $u$ is bounded on $[0,T]\times \mathcal{O}$.
Moreover, given $m$, for all $t\in[0,S_m\wedge T]$ we have
$$Y(t)=\exp\left(-\int_0^t h(X(s))\,\text{d}s\right) 
\le \exp\left(t \sup_{x\in \mathcal{O}_m}|h(x)|\right)
\le \exp\left(T\sup_{x\in \mathcal{O}_m}h(x)\right)<\infty \,,$$  
since $h$ is bounded on $\mathcal{O}_m\,$. Therefore, the family of random variables $\{Y(t)$, $t\in[0,S_m\wedge T]\}$ is uniformly bounded. 
Notice further that the paths of the processes $X(\cdot)$ and $Y(\cdot)$ are continuous, and that $\lim_{i\rightarrow \infty} \tau_{i,m} = S_m\wedge T$. Taking $i\rightarrow \infty$ in \eqref{eq:u(1/i,x_0)} yields
\begin{eqnarray*}
&& v(T,x_0)\  
=\ u(0,x_0) \ 
=\ \lim_{i\rightarrow \infty} u\big( 1/i \,,x_0\big)\
\geq\   \lim_{i\rightarrow \infty} \mathbb{E}_{x_0}\left[Y(\tau_{i,m}) \,u\left(\frac{1}{i}+\tau_{i,m},X(\tau_{i,m})\right)\right] \\
&=& \mathbb{E}_{x_0}\left[\lim_{i\rightarrow \infty} Y(\tau_{i,m})\,u\left(\frac{1}{i}+\tau_{i,m},X(\tau_{i,m})\right)\right]\,  
=\ \mathbb{E}_{x_0}\big[Y(S_m\wedge T)\,v\left(T-S_m\wedge T,X(S_m\wedge T)\right)\big]\\ 
&\geq&  \mathbb{E}_{x_0}\left[\bm{1}_{\{S_m>T\}}Y(T)f^+(X(T))\right] \
\rightarrow\ \mathbb{E}_{x_0}\left[\bm{1}_{\{S>T\}}Y(T)f^+(X(T))\right]\ 
\ge\ \mathscr{U}(T,x_0)\,, {\mathrm{\quad as\ }}m\rightarrow \infty\,,
\end{eqnarray*}
by the monotone convergence theorem, as claimed, where
$f^+(\cdot) = f(\cdot) \vee 0$\,.
\end{proof}
\begin{remark}
In the case of $\mathcal{O}$  bounded  with  $\partial \mathcal{O}$ piecewise smooth, besides the condition of Corollary \ref{coro:continuity}, we have another sufficient condition for $\,\mathscr{U}$ to be a classical solution of \eqref{eq:F-K}, \eqref{eq:f} whenever it can be checked directly that $\mathscr{U}\equiv 0$ on $(0,\infty)\times \partial \mathcal{O}\,$:  {\it the functions $a$, $b$, $f$ and $h$ are bounded and $a$ is continuously differentiable and uniformly elliptic on $\mathcal{O}$}. Under these conditions, a classical PDE result (see e.g. \cite[Theorem III.12.1]{LSU}) states that the Cauchy problem \eqref{eq:F-K}, \eqref{eq:f} has
a unique weak solution $\mathscr{V}$ (in the distributional sense) that vanishes on $(0,\infty)\times \partial \mathcal{O}\,$;  
further, we have $\mathscr{V}\in C([0,T]\times \mathcal{O})\cap C^{1,2}((0,T)\times \mathcal{O})$  for any $T\in (0,\infty)$.
One can then show that $\mathscr{V}$ coincides with $\mathscr{U}$ on $(0,\infty)\times \mathcal{O}$ by integrating \eqref{eq:dIto F-K} with respect to $t$ over $[0,S\wedge t^*]$, for any $t^*\in (0,\infty)$, and the conclusion follows. 
The result of $\mathscr{V}=\mathscr{U}$ is similar to Theorem II.2.3 and the discussion in \cite[pp.\,133--135]{F85}. 
The special case $U$ was dealt with in \cite[Theorem 2.7]{PW}. \qed
\end{remark}

\subsection{Connections between distributional and viscosity solutions}
\label{section:connection}

In \cite {I}, H.$\,$Ishii showed that the  two kinds of weak solutions of \eqref{eq:F-K}   we have been discussing are equivalent, under very nice regularity  conditions on the coefficients  as well as on the weak solution $u$\,. We state his results in our setting:

\begin{theorem}{\em{(\cite[Theorems 1 and 2]{I})}}\label{thm:Ishii}
Let $\mathcal{Q}:=(0,\infty)\times \mathcal{O}$. 
Assume that $a_{ij}\in C^{1,1}(\mathcal{Q})$ and $b_i\in C^{0,1}(\mathcal{Q})$ with $\mathcal{Q}$ regarded as an open subset of $(0,\infty)\times \mathbb{R}^n$,  and $h$ and $u$ are continuous. Then 
 
{\bf (i)} if $\,u$ is a viscosity solution of \eqref{eq:F-K}, it is also a distributional solution of \eqref{eq:F-K}; whereas 

{\bf (ii)} if $\,\sigma\in C^1(\mathcal{Q})$ and $u$ is a distributional solution of \eqref{eq:F-K}, then $u$ is also a viscosity solution of \eqref{eq:F-K}.
\end{theorem}

We have shown in Theorem \ref{thm:viscosity F-K} that $\mathscr{U}$ is a 
viscosity solution of \eqref{eq:F-K}. In order to apply Theorem \ref{thm:Ishii} to prove that $\mathscr{U}$ is also a distributional solution, we need $a_{ij}\in C^{1,1}(\mathcal{Q})$,  $b_i\in C^{0,1}(\mathcal{Q})$, and $\mathscr{U}$ to be  continuous. The difficulty in our setting comes from the fact that we want to assume the coefficients to be only continuous, and that the continuity of $\mathscr{U}$ is not known {\it a priori}. 
However, if all of the conditions in Theorem \ref{thm:Ishii} hold, we have shown that $\mathscr{U}$ is then a classical solution (see the discussion in the second paragraph of Section \ref{section:continuity}).

\section{Conclusions and future research}

We have characterized the function $\mathscr{U}$ as a viscosity solution of  of the Cauchy problem \eqref{eq:F-K}, \eqref{eq:f}, and shown that $\mathscr{U}$ is dominated by any 
nonnegative classical supersolution of this Cauchy problem. 
We have also shown that $\mathscr{U}$ is jointly continuous under appropriate conditions when $n=1$, and therefore the smallest nonnegative classical (super)solution of this Cauchy problem.
It would be interesting to investigate whether those conditions can be relaxed.

It would be also  interesting to study whether $\mathscr{U}$ is a distributional (super)solution, a requirement which consists of two parts: that $\mathscr{U}\in W^{1,2}_{loc}((0,\infty) \times \mathcal{O})$ and that \eqref{eq:F<=0 on Gamma} holds  for some $\Gamma$. The analysis in Section \ref{section:connection} shows that one cannot apply the results or the approach in \cite{I} directly.

Finally, it would be interesting to explore the joint continuity of $\mathscr{U}$ in the multidimensional case; as was  pointed  out  in the second paragraph of Section \ref{section:continuity}, with jointly continuous $\mathscr{U}$,  continuous $f$, and locally H\"{o}lder $a\,(>0)$, $b$ and $h\,(\ge 0)$, one can conclude that $\mathscr{U}$ is a classical solution to \eqref{eq:F-K}.

\appendix

%\section{Another proof of Theorem \ref{thm:viscosity} in the one-dimensional case}
\section{\texorpdfstring{Another proof of Theorem \ref{thm:viscosity} in the one-dimensional case}{Another proof of Theorem 1.1 in the one-dimensional case}}

This section  presents a proof of Theorem \ref{thm:viscosity} in the one-dimensional case which is direct, that is, does not proceed by contradiction. This proof also works in several dimensions, whenever it is known in advance that the  function $U$ of \eqref{eq:U}  is continuous.

Suppose that $\mathcal{O}=(\ell,r)$ is a fixed open interval with $-\infty \le \ell < r \le \infty$\,.
 
\begin{theorem}\label{thm:vis1}
Assume that $\sigma$ and $b$ are continuous. Then the function $ \,U $ of \eqref{eq:U} is a viscosity solution of the parabolic equation \eqref{eq:PDE1}.
\end{theorem}

\begin{proof}
Let us show that $U(t,x)$ is a viscosity subsolution to \eqref{eq:PDE1}. The proof for the viscosity supersolution property is similar.

According to the definition of viscosity subsolution and the joint continuity of $U$ (\cite[Proposition 4.3]{KR}), if suffices to verify that for any $(t_0,x_0) \in (0,\infty) \times \mathcal{O}$ and any $\varphi(t,x)\in C^{1,2}((0,\infty) \times \mathcal{O})$ with
\begin{equation}\label{eq:max1}
(U - \varphi) (t_0,x_0) = \max_{(t,x)\in (0,\infty) \times \mathcal{O}}  (U - \varphi) (t,x) = 0\,,
\end{equation}
we have the inequality
$$G(t_0,x_0)  \le 0\quad \text{for the function}\quad G(t,x):=(\varphi_t-\mathcal{L}\varphi)(t,x)$$
with the operator $\mathcal{L}$ defined in \eqref{eq:L}.

We start the diffusion process $\, X(\cdot)$ at $\,X(0) = x_0\,$, consider  $\alpha>0$ such that $[x_0-\alpha, x_0+\alpha] \subset \mathcal{O}$, and define $T_\alpha :=\inf \left\{t>0: |X(t)-x_0| > \alpha \right\}$. 
Then we have $\,S>T_\alpha$ if $\,T_\alpha < \infty\,$, and $S = \infty$ if $T_\alpha = \infty$. 
In either case, the bounded stopping time $$\theta_m \,:= \,\frac{1}{\,m\,} \wedge T_\alpha< S\,.$$ 
Further, for every fixed $\omega \in \Omega$, we notice that $\theta_m (\omega) = 1 /m$  when $m>1 / T_\alpha (\omega)$.

\smallskip
Take $m> \frac{2}{ t_0}$, then we have $\theta_m \le \frac{1}{m} \wedge S \le \frac{t_0}{2} \wedge S <t_0$\,. The strong Markov property of $X(\cdot)$ gives
\begin{eqnarray*} 
\mathbb{E}_{x_0}\left[U(t_0-\theta_m,X(\theta_m))\right]
&= &  \mathbb{E}_{x_0}\left[ \mathbb{P}_{X(\theta_m)}[S>t_0-\theta_m]\right]\ 
=\ \mathbb{E}_{x_0}\big[ \mathbb{P}_{x_0}\left[S>t_0\,|\,\mathscr{F}(\theta_m)\right]\big] \\ \nonumber
&=& \mathbb{P}_{x_0}[S>t_0]  
=U(t_0,x_0)\,.
\end{eqnarray*}

On the other hand, we recall \eqref{eq:max1} and obtain
\begin{equation} \label{eq:ineq1}
\varphi(t_0,x_0) - \mathbb{E}_{x_0}\left[\varphi(t_0-\theta_m,X(\theta_m))\right]\le U(t_0,x_0) - \mathbb{E}_{x_0}\left[U(t_0-\theta_m,X(\theta_m))\right] =0\, .
\end{equation}

Thanks to the assumption $\varphi \in C^{1,2}((0,\infty)\times \mathcal{O})$, we can apply It\^o's change of variable rule  to $\varphi(t_0-s,X(s))$, $s\in [0, \theta_m]\subset[0,t_0)$  and plug in \eqref{eq:SDE} to derive the semimartingale  decomposition  
\begin{equation}\label{eq:dIto}
\text{d} \varphi(t_0-s,X(s))= - G(t_0-s, X(s))\, \text{d}s + \varphi_{x}(t_0-s, X(s))\,\sigma(X(s))\, \text{d}W(s)\,.
\end{equation}
Integrating \eqref{eq:dIto} with respect to $s$ over $[0,\theta_m]\subset[0,t_0)$ and taking the expectation under $\mathbb{P}_{x_0}$ yields
\begin{align*}\nonumber
0\ \ge&\ \, m \cdot \text{LHS\ of\ } \eqref{eq:ineq1}\ 
=\ m\, \mathbb{E}_{x_0}\left[\int_{0}^{\theta_m} G(t_0-s, X(s))\, \text{d}s - \int_{0}^{\theta_m} \sigma(X(s))\varphi_{x}(t_0-s, X(s))\, \text{d}W(s)\right]\\ 
=&\ \, \mathbb{E}_{x_0}\left[m \int_{0}^{\theta_m} 
G(t_0-s, X(s))\, \text{d}s\right]\ :=\ G_m(t_0,x_0)\,.
\end{align*} 
Here in the second equality, the expectations of the integrals with respect to $\mathrm{d}W_k(s)$ have all vanished, due to the local boundedness of functions  $\sigma$ and $\varphi_x$ and the fact that for all $s\in [0, \theta_m]$, we have
$t_0-s\in [t_0/2,t_0]$ and $X(s) \in[x_0-\alpha, x_0+\alpha] \subset \mathcal{O}$  since $\theta_m<\frac{t_0}{2} \wedge S$.  

\smallskip
It then suffices to show that $G_m(t_0,x_0)\to G(t_0,x_0)$ as $m\rightarrow \infty$\,. 
In fact, recalling that $\theta_m \le 1/m$, as well as the boundedness of the terms in $G(t_0-s, X(s))$, we can take the limit and appeal to the dominated convergence theorem to obtain  
\begin{equation*}
\lim_{m\rightarrow \infty} G_m(t_0,x_0)= \mathbb{E}_{x_0}\left[\lim_{m\rightarrow \infty} m \int_{0}^{\theta_m}
G(t_0-s, X(s))\, \text{d}s\right].
\end{equation*}
Finally, since $\theta_m = 1/m$ $(\rightarrow 0)$ for sufficiently large $m$,  the mean value theorem and the continuity in $s$ give 
$\,\lim_{m\rightarrow \infty} G_m(t_0,x_0) = G(t_0,x_0)\,$, 
as desired. 
\end{proof}

%\section{Proof of \eqref{eq:dIto F-K}}
\section{\texorpdfstring{Proof of \eqref{eq:dIto F-K}}{Proof of (2.14)}}
\label{app:phiY}

\begin{proof}
We have
\begin{align*}
\text{LHS}\cdot (Y(t))^{-1} 
=&\ - \varphi_{t}(t^*-t,X(t))\, \text{d}t
+ \sum_{i} \varphi_{x_i}(t^*-t,X(t))\, \text{d}X_i(t)\\
\nonumber
&\ + \sum_{i,j} \frac{1}{2} \, \varphi_{x_i x_j}(t^*-t,X(t))\, \text{d}\langle X_i,X_j\rangle (t) 
- h(X(t)) \varphi(t^*-t,X(t))\, \text{d}t\\
\nonumber
=&\ - \varphi_{t}(t^*-t,X(t))\, \text{d}t
+ \sum_{i,k} \varphi_{x_i}(t^*-t,X(t))\big[\sigma_{ik}(X(t))\, \text{d}W_k(t) + b_i(X(t))\, \text{d}t\big]\\
\nonumber
&\ + \sum_{i,j,k} \frac{1}{2} \varphi_{x_i x_j}(t^*-t,X(t)) \sigma_{ik}(X(t))\sigma_{jk}(X(t))\, \text{d}t - h(X(t)) \varphi(t^*-t,X(t))\, \text{d}t\\
\nonumber
=&\ \left[- \varphi_{t}+\sum_i  b_i(X(t)) \varphi_{x_i} 
+ \sum_{i,j} \frac{1}{2} \,a_{ij}(X(t))\varphi_{x_i x_j} - h(X(t))\varphi \right](t^*-t,X(t))\, \text{d}t\\
\nonumber
&\ + \sum_{i,k} \varphi_{x_i}(t^*-t,X(t))\sigma_{ik}(X(t))\, \text{d}W_k(t)\\
\nonumber
=&\ \text{RHS}\cdot (Y(t))^{-1}. \qedhere
\end{align*}
\end{proof}

\section*{Acknowledgements}

The author is greatly indebted to her advisor Professor Ioannis Karatzas for suggesting this problem and for his very careful reading of previous versions of the paper and invaluable advice. 
She is also indebted to Professor Johannes Ruf for his very careful reading of the manuscript and his many helpful suggestions. 
She is grateful to
a referee, an area editor and the editor for their careful and helpful comments.
This work was supported partly by the National Science Foundation under grant NSF-DMS-14-05210.

\end{document}